\documentclass[a4paper]{amsart}
\usepackage{amssymb,xcolor,wasysym,xfrac,mathrsfs,diagrams}
\usepackage[utf8]{inputenc}
\usepackage[colorinlistoftodos]{todonotes}
\title{Places, cuts and orderings of function fields}
\author[P. Koprowski \and K. Kuhlmann]{Przemysław Koprowski \and Katarzyna Kuhlmann}
\address{Institute of Mathematics\\
University of Silesia\\ Bankowa 14\\ 40-007 Katowice,
Poland} \email{przemyslaw.koprowski@us.edu.pl}
\address{Institute of Mathematics\\
University of Silesia\\ Bankowa 14\\ 40-007 Katowice,
Poland} \email{kmk@math.us.edu.pl}
\newcommand{\st}{\bigm|}			
\newcommand{\restrict}{\bigm|}
\newcommand{\baseField}{\Bbbk}			
\newcommand{\kk}{\baseField}
\newcommand{\funcField}{K}			
\newcommand{\F}{\funcField}			
\newcommand{\Curve}{\mathfrak{c}}		
\newcommand{\Infty}[1]{\mathfrak{\infty}_{#1}}	
\newcommand{\csf}[1]{\eta_{#1}}			
\newcommand{\intf}[1]{\chi_{(#1)}}		
\renewcommand{\AA}{\mathbb{A}}			
\newcommand{\NN}{\mathbb{N}}			
\newcommand{\PP}{\mathbb{P}}			
\newcommand{\QQ}{\mathbb{Q}}			
\newcommand{\RR}{\mathbb{R}}			
\newcommand{\CO}{\mathcal{O}}			
\newcommand{\ordering}{\beta}			
\newcommand{\Orderings}{X}				
\newcommand{\place}{\lambda}			
\newcommand{\setP}{\mathcal{P}}			
\newcommand{\gl}{\mathfrak{l}}
\newcommand{\go}{\mathfrak{o}}
\newcommand{\gp}{\mathfrak{p}}			
\newcommand{\gq}{\mathfrak{q}}			
\newcommand{\gr}{\mathfrak{r}}			
\newcommand{\gu}{\mathfrak{u}}
\newcommand{\CL}{\mathfrak{L}}			
\newcommand{\CU}{\mathfrak{U}}			
\newcommand{\VL}{L}				
\newcommand{\VU}{U}				
\newcommand{\ball}{B}				
\newcommand{\Balls}{\mathscr{B}}		
\newcommand{\VG}{v\kk}				
\newcommand{\cto}{\Psi}				
\newcommand{\otc}{\Phi}				
\newcommand{\Ze}{\mathcal{Z}}			
\newcommand{\up}[1]{\uparrow\! #1}		
\newcommand{\lo}[1]{\downarrow\! #1}		
\newcommand{\locut}[1]{C^-_{#1}}		
\newcommand{\upcut}[1]{C^+_{#1}}		
\newcommand{\Cuts}{\mathscr{C}}			
\newcommand{\PCuts}{\mathscr{C}^*}		
\DeclareMathOperator{\dcup}{\mathbin{\dot\cup}}
\DeclareMathOperator{\conv}{conv}
\DeclareMathOperator{\id}{id}
\DeclareMathOperator{\sgn}{sgn}
\newcommand{\term}[1]{\emph{#1}}		
\numberwithin{equation}{section}
\newtheorem{thm}[equation]{Theorem}
\newtheorem{lem}[equation]{Lemma}
\newtheorem{prop}[equation]{Proposition}

\theoremstyle{definition}
\newtheorem{df}[equation]{Definition}

\begin{document}
\begin{abstract}
In this paper we investigate the space of $\mathbb{R}$-places of an algebraic function field of one variable. We deal with the problem of determining when two orderings of such a field correspond to a single $\mathbb{R}$-place. To this end we introduce and study the space of cuts on a real curve and prove that the space is homeomorphic to the space of orderings. Finally, we prove that two cuts (consequently, two orderings) correspond to a single $\mathbb{R}$-place if they are induced by a single ultrametric ball.
\end{abstract}
\maketitle

\section{Introduction}
The connection between orderings, valuations and cuts is a very important aspect of the theory of ordered fields. The following facts are well known (see for example \cite{Lam83} and \cite{Pre84}).

For an ordering $\ordering$ of a formally real field  $F$, the convex hull (with respect to $\ordering$)  of the rationals in $F$, $\conv_{\ordering}(\QQ)$, is a valuation ring of $F$. The associated valuation~$v_{\ordering}$ is called the \term{natural valuation of $\ordering$}. We shall  denote its  value group by~$v_{\ordering}F$. Note that $v_{\ordering}$ is a trivial valuation  if and only if $\ordering $ is an archimedean ordering of~$F$.  The ring $\conv_{\ordering}(\QQ)$ has a unique maximal ideal of \term{infinitesimals} and its residue field $Fv_{\ordering}$ admits an archimedean ordering induced by $\ordering$. Therefore, $Fv_{\ordering}$ can be embedded in $\RR$ and the corresponding place $\lambda_{\ordering}$ is an \term{$\RR$-place} of $F$.  The map
\begin{equation}\label{eq_lambda}
\ordering  \mapsto \lambda_{\ordering}
\end{equation}
is  surjective by the Baer-Krull Theorem. The set of all orderings of $F$, denoted by  $\Orderings(F)$,  is a topological space with a topology induced by the subbasis of \term{Harrison sets}:
\[
H(a) = \bigl\{\ordering \in \Orderings(F)\st a \in \ordering\bigr\},\qquad\text{for}\quad a\in \dot F.
\]  
By the surjectivity of the map in ~\eqref{eq_lambda}, the set of all $\RR$-places of $F$ becomes a topological space equipped with the quotient topology.  This topology is Hausdorff and compact. Any two $\RR$-places can be separated (as functions) by elements of the \term{real holomorphy ring of $F$}, which  is the intersection of all  valuation rings of $F$ with formally real residue fields. If $K$ is a field extension of $F$ and $\ordering $ is an ordering of~$K$, then $\ordering \cap F$ is an ordering of $F$. The restriction of  $\lambda_{\ordering}$ to $F$ is then equal to $\lambda_{\ordering\cap F}$ and if $v_\ordering$ is the natural valuation of $\ordering$ then $v\mid_F$ is the natural valuation of $\ordering \cap F$.

Throughout this paper, by $\kk$ we always denote an arbitrary real closed field---possibly a nonarchimedean one. The interaction between orderings  of a rational function field $\kk(x)$ and the cuts of $\kk$ is well known. Orderings of $\kk(x)$ are in one-to-one correspondence with cuts of $\kk$ (see \cite{Gil81}). Even more, the space of orderings $\Orderings\bigl(\kk(x)\bigr)$ with the Harrison topology is homeomorphic to the set $\Cuts(\kk)$ of all cuts of $\kk$ equipped with the order topology. In this particular case, the map in  \eqref{eq_lambda} depends strongly on the natural valuation $v$ of $\kk$ (i.e., on the valuation associated with the unique ordering of $\kk$). More precisely, $v$ induces an ultrametric on $\kk$  which plays a major role in the classification of cuts and places  (see \cite{KMO11} and the last section of this paper).

The goal of the present work is to generalize the above results to any algebraic function field $\F$ of transcendence degree 1 over $\kk$. The starting point is the well known interplay of  $\kk$-places  and curves  described nearly 40 years ago by M.~Knebusch (see \cite{Kne76a, Kne76b}). Recently, C.~Scheiderer (see Appendix to~\cite{GBH13}) proved that the space of orderings of $\F$ is homeomorphic to the space of orderings of $\kk(x)$. In our paper, using completely different techniques, we develop a general framework, which in particular can be used to reprove Scheiderer's result (see the comment on page~\pageref{eq_Cuts=disjoint_union} after Equation~\eqref{eq_Cuts=disjoint_union}). Our approach has the additional advantage of preserving more information about the geometric structure of the curve.

Our paper is organized as follows. First, in Section~\ref{sec_real_curves} we gather all the necessary terminology and facts concerning real algebraic curves which are needed later. 

In Section~\ref{sec_cuts}, we introduce cuts on a real curve and investigate their connections to orderings. Here we also formulate our first main result (see Theorem \ref{thm_Cuts_homeo_Orderings}), asserting that the space of orderings of $K$ is homeomorphic to the space of cuts of a curve which is a smooth projective model for $K$. For any transcendental element $x \in \F$ we define the projection $\pi_x$ of the set of cuts of the curve to the set of cuts of~$\kk$, and we show that this projection is compatible with the restriction of orderings of $K$ to the rational function field $\kk(x)$.

Next, in Section~\ref{sec_uballs}, we recall the notion of an ultrametric  on $\kk$. This ultrametric can be extended to an ultrametric of the affine space $\AA^n \kk$. 

The ultrametric balls in $\kk$ determine cuts in $\kk$ that are called \term{ball cuts}. If  $x \in \F$ is transcendental over $\kk$ and $C$ is a cut of the curve such that $\pi_x(C)$ is a ball cut, then $\pi_y(C)$ is a ball cut for any $y \in \F$ transcendental over $\kk$. We discuss this fact in Section~\ref{sec_ball_cuts}. The cuts of the curve with this property we call  \term{ball cuts} as well. 
Moreover, we give  a sufficient and necessary  condition for two orderings of a function field to correspond to a single $\RR$-place.

In the last section we embed our curve in an affine space $\AA^n \kk$. The ultrametric balls in $\AA^n \kk$ determine ball cuts on the curve. Conversely, we show that every ball cut is determined by some ultrametric ball in $\AA^n \kk$. We also prove that any two cuts of a curve which are associated with the orderings compatible with the same $\RR$-place are determined by the same ball in $\AA^n\kk$. 

\section{Real curves}\label{sec_real_curves}
The aim of this section is to gather tools and terminology concerning real algebraic curves which we shall use subsequently. Most of the results presented here are well known, hence we take the liberty to omit the proofs,  pointing the reader to appropriate bibliography instead. The standard references for the subject are \cite{BCR98, Kne76a, Kne76b}.

From now on, let $\kk$ be a fixed real closed field and $\F$ a formally real algebraic function field (of one variable) over $\kk$. Consider the set
\[
\bigl\{ \CO_\gp \st \CO_\gp\text{ is a valuation ring with maximal ideal $\gp$ such that } \kk\varsubsetneq \CO_\gp\varsubsetneq \F\bigr\}
\]
of all proper valuation rings of $\F$ containing $\kk$. Every such  valuation ring $\CO_\gp$ is uniquely determined by its maximal ideal $\gp$. These ideals may be regarded as closed points of the scheme associated with $\F$. Following \cite{Kne76a}, we say that a point $\gp$ is \term{real} if the residue field $\sfrac{\CO_\gp}{\gp}$ is formally real, hence equal to $\kk$ (i.e., $\CO_\gp$ is residually real in terms of \cite{Lam83}). The set of all real points is a complete smooth real algebraic curve that will be denoted $\Curve$ in this paper. We treat elements of $\F$ as functions on~$\Curve$.

The real closed field $\kk$ has a natural topology induced by its (unique) ordering. This topology extends to the affine space $\AA^n\kk$ and further to the projective space~$\PP^n\kk$. Every embedding of $\Curve$ in $\PP^n\kk$ induces a topology on $\Curve$. Following \cite{BCR98}, we call it the \term{euclidean topology} (note that in \cite{Kne76a, Kne76b} this is called the \term{strong topology}). This is the coarsest topology with respect to which all functions from $\F$ are continuous.

Recall (see e.g.\ \cite[Theorem~2.4.4]{BCR98}) that $\Curve$ is a disjoint union of finitely many semi-algebraically connected components $\Curve_1, \dotsc, \Curve_N$. These components can be separated by functions from $\F$ (see \cite[Theorem~2.10]{Kne76a}) in the sense that for every $\Curve_i$ there exists $\csf{i}\in K$ such that
\[
\sgn \csf{i} (\gp) = 
\begin{cases}
1 & \text{if } \gp\in \Curve\setminus \Curve_i,\\
-1 & \text{if } \gp\in \Curve_i.
\end{cases}
\]
The functions $\csf{i}$ are called \term{component separating functions} and are determined uniquely up to  multiplication by sums of squares. Each component is homeomorphic to the projective line $\PP^1\kk$ (or equivalently, to a circle over $\kk$),
hence it admits two orientations. Consequently, $\Curve$ admits a total of $2^N$ possible orientations. Any such orientation is uniquely determined by a definite differential on $\Curve$. In fact, orientations may be identified with equivalence classes of definite differentials (see \cite[\S5]{Kne76a}). 
\begin{quote}
\emph{We assume that an orientation of $\Curve$ is arbitrarily chosen and fixed.}
\end{quote}
This lets us define intervals on $\Curve$ (see \cite[\S6]{Kne76a}). Let $\gp, \gq$ be two points belonging to the same component $\Curve_i$ of $\Curve$. An open interval $(\gp, \gq)$ consists of all the points $\gr\in \Curve$ that lie properly between $\gp$ and $\gq$ with respect to the fixed orientation. We shall need the following two results:

\begin{thm}[{\cite[Theorem~3.4]{Kne76a}}]\label{thm_even_zeros}
For every $f\in \F$ and every component $\Curve_i$ of $\Curve$, the number of points $\gp\in \Curve_i$ at which $f$ changes sign is finite and even.
\end{thm}

\begin{thm}[{\cite[Theorem~4.5]{Kne76a}}]\label{thm_function_from_points}
Given an even number of points in each component $\Curve_i$ of $\Curve$, there exists a function $f\in \F$ which changes sign precisely at these points.
\end{thm}

In particular, it follows that for every interval $(\gp, \gq)\subset \Curve_i$, there is a function $\intf{\gp,\gq}\in \F$ satisfying:
\[
\sgn \intf{\gp,\gq}(\gr) =
\begin{cases}
1 & \text{if } \gr\notin [\gp, \gq],\\
0 & \text{if } \gr \in \{ \gp, \gq\},\\
-1 & \text{if } \gr\in (\gp, \gq).
\end{cases}
\]
This function is called an \emph{interval function} of $(\gp, \gq)$ and---like $\csf{i}$ above---it is unique up to  multiplication by  sums of squares. In particular, it is safe to talk about the sign of $\intf{\gp,\gq}$ or $\csf{i}$ with respect to an ordering of $\F$.

We will need the following, rather basic, fact. Unfortunately, we are not aware of any convenient reference, hence we feel obliged to prove it explicitly.

\begin{prop}\label{prop_piecewise_monotonic}
For every nonconstant function $f \in \F$ and every component $\Curve_i\subseteq \Curve$ there are finitely many points $\gp_0, \dotsc, \gp_m\in \Curve_i$ such that $f$ is  monotonic and has no poles on every interval between two consecutive points.
\end{prop}

\begin{proof}
Recall that an orientation of $\Curve$ is fixed. Let $\omega_0$ be a differential of $\F$ associated with this orientation. Then there is a unique element $f'\in \F$ such that
\[
df = f'\cdot \omega_0.
\]
Fix a component $\Curve_i\subseteq \Curve$ and let $\setP$ be the set of all the points $\gp\in \Curve_i$ where either $f$ has a pole or $f'$ has a pole or
$f'$ changes the sign. Every nonzero function has a finite number of poles; together with Theorem~\ref{thm_even_zeros} this shows that $\setP$ consists of a finite  number of points, say, $\setP = \{ \gp_0, \dotsc, \gp_m\}$. Assume that the indexation of the $\gp_i$'s is coherent with the orientation of $\Curve$ (i.e., the set $\Curve_i\setminus \{\gp_0\}$ is totally ordered by the orientation and $\gp_1\prec \gp_2\prec \dotsb \prec \gp_m$ holds with respect to this ordering). It follows from \cite[Theorem~6.8]{Kne76a} that $\sgn_\gq f'$ is constant and non-zero for every $\gq\in (\gp_j, \gp_{j+1})$ and every $j\in \{0, \dotsc, m\}$ (here, for simplicity, we use indexation modulo $m+1$, so that $\gp_{m+1} = \gp_0$). Now, \cite[Corollary~8.4]{Kne76b} asserts that $f$ is strictly increasing on $(\gp_j, \gp_{j+1})$ when $\sgn_\gq f' = 1$ for $\gq \in (\gp_j, \gp_{j+1})$, and strictly decreasing otherwise. This concludes the proof.
\end{proof}

Recall (see e.g.\ \cite[Chapter~2]{Lam83}) that an ordering $\ordering$ of $\F$ is said to be compatible with a valuation ring $\CO_\gp$, or shortly with a point $\gp$, if $1+\gp \subseteq \ordering$. Take $\gp\in \Curve_i$. It is straightforward to check that the following subsets of $\F$ are orderings compatible with $\gp$:
\begin{equation}\label{eq_compatible_orderings}
\begin{split}
\ordering^+_\gp &= \Bigl\{ f\in \F\st \exists_{\gq\in \Curve_i} \forall_{\gr\in (\gp, \gq)} f(\gr) >0 \Bigr\},\\
\ordering^-_\gp &= \Bigl\{ f\in \F\st \exists_{\gq\in \Curve_i} \forall_{\gr\in (\gq, \gp)} f(\gr) >0 \Bigr\}.
\end{split}
\end{equation}
It follows from  the  Baer-Krull theorem that these are the only orderings compatible with $\gp$. We call them \term{principal} and we denote
\[
\Orderings_{princ}(\F):=\big\{\ordering^+_\gp, \ordering^-_\gp\st \gp \in \Curve\bigr\}.
\] 
The following important fact holds:
\begin{thm}[{\cite[Theorem~9.9]{Pre84}}]\label{thm_density}
The set $\Orderings_{princ}(\F)$ of principal orderings is dense in the set $\Orderings(\F)$ of all orderings. 
\end{thm}

From this theorem we infer:

\begin{prop}\label{density_strict}
Suppose that $f_1,\dotsc ,f_n \in \F$ and $H(f_1) \cap\dotsb \cap H(f_n) \neq \emptyset$. Then there is $\gp_0 \in \Curve$ such that $f_k(\gp_0) >0$ for $k=1,\dotsc ,n$.
\end{prop}
\begin{proof}
By Theorem \ref{thm_density} there is a principal ordering $\ordering^+_\gp$ or $\ordering^-_\gp$ in $H(f_1) \cap\dotsb \cap H(f_n)$. Consider the case of $\ordering^+_\gp$ (for $\ordering^-_\gp$ the proof is symmetric). Assume that $\gp \in \Curve_i$. By (\ref{eq_compatible_orderings}), for every $f_k$ there is $\gq_k\in \Curve_i$, $\gq_k \neq \gp$ and such that $f_k$ is positive on the interval $(\gp,\gq_k)$. Among the  $\gq_k$'s there is one, say $\gq$, such that no other $\gq_k$ lies between $\gp$ and $\gq$. Then we have that $(\gp,\gq) \neq \emptyset$ and that each $f_k$ is positive on the interval $(\gp,\gq)$. Any point 
 $\gp_0 \in (\gp,\gq)$ has the required property.
\end{proof}

\begin{prop}
For every ordering $\ordering$ of $\F$ there is exactly one component $\Curve_i$ of~$\Curve$ such that $\csf{i}\in -\ordering$.
\end{prop}

\begin{proof}
By the definition of the component separating function, for every $\gp \in \Curve$ we have $\csf{1}\dotsm \csf{N}(\gp)<0$. Therefore the function $s := -\csf{1}\dotsm \csf{N}$ is positive definite and by \cite[Theorem~4.1]{Kne76a} is a (nonzero) sum of squares in $\F$. Consequently it belongs to $\ordering$. Therefore, $-\csf{i} \in \ordering$ for at least one $i$.
 
Now assume that $-\csf{i} \in \ordering$ and $-\csf{j} \in \ordering$ for some $i \neq j$.  Then $H(-\csf{i}) \cap H(-\csf{j}) \neq \emptyset$ and, by Proposition~\ref{density_strict}, there is $\gp \in  \Curve$ such that $\csf{i}(\gp)> 0$ and $\csf{j}(\gp)> 0 $, a contradiction to the definition of the components separating functions. 
\end{proof}

The component $\Curve_i$ of the above lemma is said to be \term{associated} with the ordering~$\ordering$ of~$\F$.

\section{Cuts on real curves}\label{sec_cuts}
Recall that a \term{cut} of a totally ordered set $(A, \preceq)$ is a pair $(L, U)$ consisting of two subsets of~$A$ satisfying: $l\prec u$ for every $l\in L$ and $u\in U$. For any subset $S$ of a totally ordered set $A$ we denote:
\[
\lo S = \{a\in A \st  \exists_{s \in S}\ a\preceq s\},\qquad
\up S = \{a\in A \st  \exists_{s \in S}\ a \succeq s\}.
\]
Observe that for any $S$, taking $\up S$ as an upper cut set determines a cut in $A$ (the lower cut set is its complement). In a similar way, $\lo S$ determines a cut  as a lower cut set.

Consider a real curve $\Curve$ and its function field $\F$. In each component $\Curve_i$ of $\Curve$ we fix one point and denote it by $\Infty{i}$. The set $\Curve_i\setminus \{\Infty{i}\}$ is linearly ordered by the (fixed) orientation of the curve. Consequently, we may freely talk about cuts on $\Curve_i$. Explicitly, a \term{cut of a component} $\Curve_i$ is a pair $(\CL, \CU)$ of subsets $\CL, \CU\subset \Curve_i$ 
such that
\begin{itemize}
\item $\Curve_i$ is the disjoint union $\CL\,\dcup\; \CU\,\dcup\; \{\Infty{i}\}$, and
\item for every $\gl\in \CL$ and every $\gu\in \CU$, the point $\Infty{i}$ lies in the interval $(\gu, \gl)$.
\end{itemize}

Observe that for $\CL = \Curve_i\setminus \{\Infty{i}\}$ and $\CU = \emptyset$ (or symmetrically for $\CL = \emptyset$ and $\CU = \Curve_i\setminus \{\Infty{i}\}$) both conditions hold: the first one trivially and the second one vacuously. Thus $\bigl( \Curve_i\setminus \{\Infty{i}\}, \emptyset\bigr)$ and $\bigl( \emptyset, \Curve_i\setminus \{\Infty{i}\}\bigr)$ are admissible cuts on~$\Curve_i$. 

The set $\Curve_i\setminus \{\Infty{i}\}$ is totally ordered by the relation 
\[
\gp \prec_i \gq \Longleftrightarrow \Infty{i} \in (\gq, \gp).
\]
Therefore, the second condition of the above definition reads as: $\gl\prec_i  \gu$ for every $\gl\in \CL$ and every $\gu\in \CU$.

\begin{df}\label{df_principal_cut}
Let $\gp$ be a point of $\Curve_i$. The cuts
\begin{itemize}
\item $\upcut{\gp}$ with the lower cut set $\lo\{\gp\}$, if $\gp\neq \Infty{i}$,
\item $\locut{\gp}$ with the upper cut set $\up\{\gp\}$,
if $\gp\neq \Infty{i}$,
\item $\upcut{\Infty{i}}$ with the lower cut set $\emptyset$,
\item $\locut{\Infty{i}}$ with the upper cut set $\emptyset$,
\end{itemize}
are called \term{principal} cuts associated with the point $\gp$.
\end{df}

Explicitly, if $\gp\neq \Infty{i}$,
\[
\locut{\gp} = \bigl( (\Infty{i}, \gp), [\gp, \Infty{i})\bigr)
\qquad\text{and}\qquad 
\upcut{\gp} = \bigl( (\Infty{i}, \gp], (\gp, \Infty{i})\bigr)
\]

\begin{prop}\label{prop_ordering->cut}
Every ordering $\ordering$ of $\F$ defines a cut on the associated component~$\Curve_i$ by the formula:
\begin{equation}\label{eq_ordering->cut}
\otc(\ordering) = (\CL,\CU) \text{ with }
\begin{cases}
\CU = \bigl\{ \gp\in \Curve_i\setminus \{\Infty{i}\} \st \intf{\gp, \Infty{i}}\in \ordering\bigr\}\\
\CL = \bigl\{ \gp\in \Curve_i\setminus \{\Infty{i}\} \st \intf{\Infty{i}, \gp}\in \ordering\bigr\}.
\end{cases}
\end{equation}
\end{prop}

\begin{proof}
Take any point $\gp\in \Curve_i \setminus \{\Infty{i}\}$. An interval function $\intf{\gp, \Infty{i}}$ is either positive or negative with respect to the ordering $\ordering$. In the first case clearly $\gp \in \CU$. In the second case $\intf{\gp, \Infty{i}}\in -\ordering$ and so $\csf{i}\intf{\gp,\Infty{i}}\in \ordering$. The interval function  $\intf{\gp,\Infty{i}}$ is positive on $\Curve \setminus \Curve_i$ and on the interval $(\Infty{i},\gp)$, while  $\csf{i}$ is positive on $\Curve \setminus \Curve_i$ and negative on $\Curve_i$. Therefore, $\csf{i}\intf{\gp,\Infty{i}}$ is negative only on the interval $(\Infty{i},\gp)$. Thus there exists a sum of squares $s$ such that $$\intf{\Infty{i},\gp}= s\cdot\csf{i}\cdot\intf{\gp,\Infty{i}} \in \ordering.$$  It follows that $\gp\in \CL$.

Now we show that every element of $\CL$ precedes every element of $\CU$. Take $\gl \in \CL$ and $\gu \in \CU$. We have $\intf{\gu,\Infty{i}} \in \ordering$ and $\intf{\Infty{i},\gl} \in \ordering$. Since $-\csf{i} \in \ordering$, we obtain that for some sums of squares $s_1$ and $s_2$, 
\begin{align*}
-\intf{\Infty{i},\gu} &= -\csf{i}\cdot s_1\cdot\intf{\gu,\Infty{i}} \in \ordering,
\intertext{and}
-\intf{\gl, \Infty{i}} &= -\csf{i}\cdot s_2\cdot\intf{\Infty{i},\gl} \in \ordering.
\end{align*}
It follows from Proposition~\ref{density_strict} that there is a point $\gp \in  \Curve$ such that $\intf{\gl,\Infty{i}}(\gp)< 0$ and $\intf{\Infty{i},\gu}(\gp) < 0$. This implies that $\gp \in (\gl, \Infty{i}) \cap (\Infty{i}, \gu)$ and consequently $\Infty{i} \in (\gu, \gl)$.
\end{proof}

Conversely, we show that to every cut of $\Curve$ one may associate an ordering on $\F$.

\begin{prop}\label{prop_cut->ordering}
Every cut $(\CL,\CU)$ of $\Curve_i$ defines an ordering $\ordering$ of the field $\F$ by the formula:
\begin{equation}\label{eq_cut->ordering}
\cto\bigl((\CL,\CU)\bigr) = \Bigl\{ f\in \F\st \exists_{\gl\in \CL \cup \{\Infty{i}\}}\exists_{\gu\in \CU \cup \{\Infty{i}\}}\forall_{\gp\in (\gl,\gu)}\ f(\gp)>0\Bigr\}.
\end{equation}
\end{prop}

\begin{proof}
In order to prove that $\ordering$ is an ordering of $\F$ we need to prove four axioms. Three of them, namely: $\ordering + \ordering\subseteq \ordering$ and $\ordering\cdot \ordering\subseteq \ordering$ and $\ordering \cap -\ordering = \emptyset$ are straightforward. The remaining one is $\ordering\cup (-\ordering)= \dot \F$. Take a non-zero $f\in \F$. It has only finitely many zeros and poles on $\Curve$. Let $\Ze$  be the set of all of them and define $\gl$ to be the maximal element of $\Ze \cap \CL$ if this set is nonempty, and  $\gl = \Infty{i}$ otherwise. Similarly, define $\gu$ to  be the minimal  element of $\CU \cap \Ze$ if this set is nonempty, and $\gu = \Infty{i}$ otherwise.  If $\Ze \neq \emptyset$ then $\gl \neq \gu$. It is clear that the sign of $f$ is constant on $(\gl,\gu)$. Thus, either $f\in \ordering$ or $-f\in \ordering$. If $\Ze = \emptyset$, then  the sign of $f$ is constant on the whole component $\Curve_i$. In this case we can pick up any $\gl \in \CL$ and $\gu \in \CU$ to show that $f\in \ordering$ or $-f\in \ordering$.
\end{proof}

Denote by $\Cuts = \Cuts(\Curve)$ the set of cuts on $\Curve$. Our goal now is to prove that the functions 
\[
\otc: X(\F) \rightarrow  \Cuts(\Curve)
\qquad\text{and}\qquad
\cto: \Cuts(\Curve) \rightarrow X(\F)
\]
are inverses of each other. We first show that this holds for the principal cuts and orderings.

\begin{lem}\label{lem_prinicpal}
The function $\otc$  induces a bijection between the set $\Orderings_{\text{princ}}(\F)$ of principal orderings of $\F$ and the set $\Cuts_{\text{princ}}(\Curve)$ of principal cuts of $\Curve$. Its inverse is~$\cto$.
\end{lem}

\begin{proof}
 Take a principal ordering $\ordering^+_\gp$ associated with some $\gp\in \Curve_i$. Compute the corresponding cut $\otc(\ordering^+_\gp) =  (\CL, \CU)$. Proposition~\ref{prop_ordering->cut} asserts that
\begin{align*}
\CU 
&= \bigl\{ \gu\in \Curve_i\setminus \{\Infty{i}\}\st \intf{\gu, \Infty{i}}\in \ordering^+_\gp \bigr\}\\
&= \bigl\{ \gu\in \Curve_i\setminus \{\Infty{i}\}\st \exists_{\gq\in \Curve_i} \forall_{\gr\in (\gp, \gq)} \intf{\gu, \Infty{i}}(\gr) > 0 \bigr\}\\
&= \bigl\{ \gu\in \Curve_i\setminus \{\Infty{i}\}\st \exists_{\gq\in \Curve_i} (\gp, \gq) \subseteq (\Infty{i}, \gu) \bigr\}\\
&= \left\{\begin{array}{ll}
(\gp, \Infty{i}) &\text{ if } \gp \neq \Infty{i},\\
\Curve_i\setminus \{\Infty{i}\}& \text{ if } \gp = \Infty{i}.
\end{array}\right.
\end{align*}
Consequently $\CL = (\Infty{i}, \gp]$ and so $\otc(\ordering^+_\gp) = \upcut{\gp}$. Analogously one shows that $\otc(\ordering^-_\gp) = \locut{\gp}$.

Conversely, take a principal cut $\upcut{\gp} = (\CL, \CU)$ and compute the associated ordering $\ordering = \cto(\upcut{\gp})$. We have that $\CL \cup \{\Infty{i}\} = [\Infty{i},\gp]$ and $\CU \cup \{\Infty{i}\} = (\gp,\Infty{i}]$, hence by the definition of $\cto$,
\[
\ordering = 
\bigl\{ f\in \F\st \exists_{\gl\in [\Infty{i},\gp]} \exists_{\gu\in (\gp, \Infty{i}]} \forall_{\gq\in (\gl, \gu)}\ f(\gq) > 0 \bigr\}.
\]
Observe that the condition
$
\exists_{\gl\in [\Infty{i},\gp]} \exists_{\gu\in (\gp, \Infty{i}]} \forall_{\gq\in (\gl, \gu)}\ f(\gq) > 0 
$
is satisfied if and only if the condition 
$
\exists_{\gu\in \Curve_i} \forall_{\gr\in (\gp, \gu)}\ f(\gr) > 0 
$ 
is satisfied. Hence $\ordering= \ordering^+_\gp$, as desired. Again, the case of $\locut{\gp}$ is  analogous.
\end{proof}

In order to extend this result to all cuts and orderings we will need some technical preparation.

\begin{lem}\label{lem_1}
Let $\ordering$ be an ordering of $\F$ with the associated component $\Curve_i$. Take $\gp \neq \gq$ in $\Curve_i\setminus \{\Infty{i}\}$ and assume that $\sgn_\ordering(\intf{\gp, \Infty{i}}) = \sgn_\ordering(\intf{\Infty{i}, \gq})$. Then 
\[
\sgn_\ordering(\intf{\gp, \gq})=1 \iff \Infty{i} \in (\gp,\gq).
\]
\end{lem}

\begin{proof}
We have either $\Infty{i} \in (\gp,\gq)$  which is equivalent to $\gq \prec \gp$, or $\Infty{i} \in (\gq,\gp)$ which is equivalent to $\gp \prec \gq$.

Assume that $\Infty{i} \in (\gp,\gq)$. Consider the function 
\[
\kappa = \intf{\gp, \Infty{i}}\cdot \intf{\Infty{i},\gq}\cdot \intf{\gp,\gq}.
\]
This function takes only non-negative values on $\Curve$, so $\kappa$ is a sum of squares in~$\F$. Since $\intf{\gp, \Infty{i}}\cdot \intf{\Infty{i},\gq}$ is in $\ordering$ by assumption, also $\intf{\gp,\gq}$ is in $\ordering$. Thus $\sgn_\ordering(\intf{\gp, \gq})=1$.

On the other hand, if  $\Infty{i} \in (\gq,\gp)$, consider the function $\csf{i}\kappa$. Also this function takes non-negative values on $\Curve$, hence again is a sum of squares in~$\F$.  Therefore,  $\csf{i}\intf{\gp,\gq}$ is in $\ordering$. Since $- \csf{i}$ is in $\ordering$ (as $\Curve_{i}$ is associated with $\ordering$ by assumption), we obtain that $\sgn_\ordering(\intf{\gp, \gq})=-1$.
\end{proof}

\begin{prop}\label{prop_otc_is_injective}
For every ordering  $\ordering$  of $\F$, 
\[
(\cto \circ \otc) (\ordering)=\ordering.
\]
\end{prop} 

\begin{proof}
Let $\Curve_i$ be  the associated component for $\ordering$, $(\CL,\CU) = \otc(\ordering)$, and $\ordering' = \cto\bigl((\CL,\CU)\bigr)$. The proposition is proved  for principal  orderings, therefore we may assume that $\ordering \notin \{ \ordering^-_{\Infty{i}},\ordering^+_{\Infty{i}}\}$, hence  both  $\CL$ and $\CU$ are non-empty. 

Since $\ordering$ and $\ordering'$ are two orderings, it suffices to show that $\ordering\subseteq \ordering'$. Take $f \in \ordering$. We have that $f \in \ordering$ iff $\frac{f}{1+f^2}  \in \ordering$, and it is well known that the function $\frac{f}{1+f^2}$ is an element of the real holomorphy ring of $K$ (see \cite[Lemma~9.5]{Lam83}). Therefore, without loss of generality,  we can assume that $f$ is in the real holomorphy ring of~$\F$, so $f$ has no poles on $\Curve$. 

Let $\Ze_i(f)$ be the set of zeros of $f$ on $\Curve_i$. Let $\gp$ be the maximal element of $\CL \cap \Ze_i(f)$ if this set is non-empty, and $\gp \in \CL$ arbitrary otherwise. Let $\gq$  be the minimal  element of $\CU \cap \Ze_i(f)$ if this set is non-empty, and $\gq \in \CU$ arbitrary otherwise. Then the function $f$ has a constant sign on the interval $(\gp,\gq)$. By the definition of the cut $(\CL,\CU) = \otc(\ordering)$  we have $\intf{\gq, \Infty{i}} \in \ordering$ and $\intf{\Infty{i}, \gp}\in \ordering$.  Since $\Infty{i} \in (\gq,\gp)$, Lemma~\ref{lem_1} asserts that $\intf{\gq, \gp} \in \ordering$. Since $-\csf{i} \in \ordering$ we obtain for some sum of squares $s$ that 
\[
-\intf{\gp, \gq}= -\csf{i}\cdot s\cdot\intf{\gq,\gp} \in \ordering.
\]
Therefore $\ordering \in H(f) \cap H(-\intf{\gp, \gq})$. By Proposition~\ref{density_strict} we obtain a point $\gp_0 \in \Curve$ such that $f(\gp_0)> 0 $ and $\intf{\gp, \gq}(\gp_0)< 0$, hence $\gp_0 \in [\gp, \gq]$ and so $f$ is positive on  $(\gp,\gq)$ and it follows that $f \in \ordering'$, by the definition of $\ordering' = \cto\bigl((\CL,\CU)\bigr)$.
\end{proof}

Any two cuts $C_1 = (\CL_1, \CU_1)$ and $C_2 = (\CL_2, \CU_2)$ of the same component $\Curve_i$ can be compared by inclusion of their lower cut sets: we say that $C_1\le C_2$ if $\CL_1\subseteq\CL_2$. This way the set of cuts on $\Curve_i$ is totally ordered with the smallest element $\upcut{\Infty{i}} = \bigl( \emptyset, \Curve_i\setminus \{\Infty{i}\}\bigr)$ and the largest element $\locut{\Infty{i}} = \bigl( \Curve_i\setminus \{\Infty{i}\}, \emptyset\bigr)$. This relation induces an order topology on $\Curve_i$ and further a union topology on $\Curve$. 

\begin{lem}\label{lem_2}
Let $(\gp,\gq)$ be an interval on $\Curve_i$ such that $\Infty{i} \notin (\gp,\gq)$. Let $\ordering$ be an ordering of $\F$ with $\sgn_\ordering(\intf{\gp, \gq})=-1$. Then
\[
\upcut{\gp} \preceq \otc(\ordering)\preceq \locut{\gq}.
\]
\end{lem}

\begin{proof}
Write $\otc(\ordering)=(\CL,\CU)$. Proposition \ref{prop_otc_is_injective} asserts that $\ordering = (\cto\circ \otc)(\ordering)= \cto\bigl((\CL,\CU)\bigr)$. We show that $\upcut{\gp} \preceq \otc(\ordering)$. This is true if $\gp = \Infty{i}$. Assume that $\gp \neq \Infty{i}$ and $\otc(\ordering)\prec \upcut{\gp}$. That means that $\gp \in \CU$. Since $\Infty{i} \notin (\gp,\gq)$, we have that  $\intf{\gp, \gq}$ has positive values  on the interval $(\Infty{i},\gp)$. By the definition of $\ordering = \cto\bigl((\CL,\CU)\bigr)$, we obtain that $\intf{\gp, \gq} \in \ordering$, a contradiction. Analogously one shows that  $\otc(\ordering)\preceq \locut{\gq}.$
\end{proof}

Recall that for every real closed field $\kk$ there is a canonical homeomorphism between the space $\Orderings\bigl(\kk(x)\bigr)$ of orderings of the rational function field (equipped with the Harrison topology) and the space of  cuts of $\kk$. Our first main result generalizes this correspondence to algebraic function fields.

\begin{thm}\label{thm_Cuts_homeo_Orderings}
The space $\Cuts(\Curve)$ of cuts on $\Curve$ is homeomorphic to the space $\Orderings(K)$ of orderings of $K$. This homeomorphism can be chosen to  map principal cuts to principal orderings.
\end{thm}

\begin{proof}
Proposition~\ref{prop_otc_is_injective} asserts that $\cto\circ \otc= \id_{\Orderings(\F)}$. We claim that also $\otc\circ \cto = \id_{\Cuts(\Curve)}$. Take any cut $C = (\CL, \CU)$ of a component $\Curve_i\subseteq \Curve$. Let $\ordering := \cto(C)$ and 
$C' = (\CL', \CU'):= \otc(\ordering)$. By the definitions of $\cto$ and $\otc$, 
\begin{align*}
\CU'
&= \bigl\{ \gq\in \Curve_i\setminus \{\Infty{i}\}\st \intf{\gq, \Infty{i}}\in \ordering \bigr\}\\
&= \bigl\{ \gq\in \Curve_i\setminus \{\Infty{i}\}\st 
\exists_{\gl\in \CL\cup \{\Infty{i}\}} \exists_{\gu\in \CU\cup \{\Infty{i}\}} \forall_{\gp\in (\gl, \gu)} \ \intf{\gq, \Infty{i}}(\gp) > 0 \bigr\}\\
&= \bigl\{ \gq\in \Curve_i\setminus \{\Infty{i}\}\st 
\exists_{\gl\in \CL\cup \{\Infty{i}\}} \exists_{\gu\in \CU\cup \{\Infty{i}\}}\ (\gl, \gu)\subseteq (\Infty{i}, \gq) \bigr\}.
\end{align*}
Take any $\gu\in \CU$ and $\gl\in \CL$. Since  $(\gl, \gu)\subset (\Infty{i}, \gu)$, we obtain that $\gu\in \CU'$, so $\CU\subseteq \CU'$. On the other hand, $\gl\notin \CU'$ because otherwise there would exist $\gu\in \CU\cup \{\Infty{i}\}$ satisfying $\gu\prec \gl$. But this contradicts the very definition of a cut. Consequently, $\CL$ is disjoint from $\CU'$ and it follows that $C = C'$. This proves our claim.

We shall now prove that  $\otc$ is continuous. Observe that the set
\[
\bigl\{ (C_1, C_2), [\upcut{\Infty{i}}, C_2), (C_1,\locut{\Infty{i}}]\st C_1, C_2\in \Cuts(\Curve_i)\bigr\}
\]
is a subbasis of the topological space $\Cuts(\Curve_i)$. Fix a subbasic set $(C_1, C_2) \subset \Cuts(\Curve_i)$, with $C_1 = (\CL_1, \CU_1)$ and $C_2 = (\CL_2, \CU_2)$. Take a cut $C = (\CL, \CU)$. Note that if $\gp \in  \CL$ and $\gq \in \CU$, then the function $-\intf{\gp,\gq}$ is positive on $(\gp,\gq)$, in particular $\otc^{-1}(C)$ belongs to the Harrison set $H(-\intf{\gp,\gq})$. If $C\in (C_1, C_2)$, then we can choose $\gp \in \CU_1 \cap \CL$ and $\gq \in \CU \cap \CL_2$. That shows 
\[
\otc^{-1}\bigl((C_1, C_2)\big) \subseteq 
\bigcup_{\gp,\gq \in  \CU_1 \cap \CL_2, \gp\prec \gq} H(-\intf{\gp,\gq}).
\]
Conversely, take any two distinct points $\gp,\gq \in  \CU_1 \cap \CL_2$, say $\gp\prec \gq$. Further let $\ordering \in  H(-\intf{\gp,\gq})$  be an ordering. We have  $\sgn_\ordering(\intf{\gp, \gq})=-1$, thus Lemma~\ref{lem_2} and the choice of $\gp$ and $\gq$ yield  
\[
C_1\prec \upcut{\gp} \preceq \otc(\beta) \preceq \locut{\gq} \prec C_2.
\]
We have, thus, shown that 
\[
\otc^{-1}\bigl((C_1, C_2)\big) = 
\bigcup_{(\gp,\gq) \subset  \CU_1 \cap \CL_2} H(-\intf{\gp,\gq}).
\]
In a similar way one shows that 
\[
\otc^{-1}\bigl([\Infty{i}, C_2)\big) = 
\bigcup_{\gq \in \CL_2} H(-\intf{\Infty{i},\gq}), 
\]
and 
\[
\otc^{-1}\bigl((C_1, \Infty{i}]\big) = \bigcup_{\gp \in \CL_1} H(-\intf{\gp,\Infty{i}}), 
\]
Inverse images of subbasic sets are open in the Harrison topology of the space~$X(K)$. The space $X(K)$ is compact and $\Cuts(\Curve)$ is Hausdorff. A continuous bijection of a compact space onto a Hausdorff space is a homeomorphism.

The last assertion of our theorem follows  from Lemma~\ref{lem_prinicpal}.
\end{proof}

 Take a component $\Curve_i\subseteq \Curve$ and $\gp, \gq \in\Curve_i$, $\gp \neq \gq$. Assume that the interval $(\gp,\gq)$ does not contain~$\Infty{i}$. Then the interval $(\gp,\gq)$ is an ordered set with the ordering inherited from $\Curve_i \setminus \{\Infty{i}\}$.
We wish to identify the cuts of this ordered set with  cuts of $\Curve_i$ which induce them. We set
$\PCuts\bigl( (\gp, \gq)\bigr) = $
$$\bigl\{(\CL,\CU) \in \Cuts(\Curve_i) \mid \,
\CL \cap (\gp,\gq) \neq \emptyset \text{ and } \CU \cap (\gp,\gq) \neq \emptyset  \text{ or } \CU =  \up (\gp,\gq)  \text{ or } \CL = \lo (\gp,\gq)\bigr\}.$$

Take the finite sequence of points of $\Curve_i$, i.e.,
$ \gp_1\prec \gp_2\prec \dotsb \prec \gp_m$.  Then $\Curve_i$ can be expressed as a finite union
 \[
\Curve_i = \{\Infty{i}\} \,\dcup\; (\Infty{i}, \gp_1) \,\dcup\; \{\gp_1\}\,\dcup\; \dotsb
\,\dcup\;  (\gp_{m-1}, \gp_m) \,\dcup\; \{\gp_m\} \,\dcup\; (\gp_m, \Infty{i})
\]
of disjoint subsets (intervals and their endpoints).
Then the set $\Cuts(\Curve_i)$ of all  cuts of the component $\Curve_i$ can be expressed as the disjoint union
\begin{equation}\label{eq_Cuts=disjoint_union}
\Cuts(\Curve_i) = \PCuts\bigl( (\Infty{i}, \gp_1)\bigr) \,\dcup\; \dotsb 
\,\dcup\; \PCuts\bigl( (\gp_{m-1}, \gp_m)\bigr)
\,\dcup\; \PCuts\bigl( (\gp_{m}, \Infty{i})\bigr).
\end{equation}

Fix an element $x\in \F\setminus \kk$. It is transcendental over $\kk$ and $\kk(x)\subseteq \F$.
Proposition~\ref{prop_piecewise_monotonic} asserts that
we can choose a sequence $\gp_1\prec \gp_2\prec \dotsb \prec \gp_m$ in such a
way that $x$ is monotonic and has no poles on each of the above intervals. 

Take an interval $(\gp, \gq)$ on which $x$ is monotonic and without poles. By the Intermediate Value Theorem  (see \cite[Theorem 8.2]{Kne76b}), the projection  $\gr\mapsto x(\gr)$ is an order isomorphism  of $(\gp, \gq)$ onto  the interval $I:=\bigl(x(\gp),x(\gq)\bigr)$ or $I:=\bigl(x(\gq),x(\gp)\bigr)$ in $\kk$. This order isomorphism   induces an order isomorphism $\pi_x$ from  $\PCuts\bigl((\gp, \gq)\bigr)$  onto  $\Cuts(I)$, and hence also  onto $\PCuts(I) \subset \Cuts(\kk)$.

By Equation~\eqref{eq_Cuts=disjoint_union} we obtain a map 
\[
\pi_x: \Cuts(\Curve)\to \Cuts(\kk).
\]

Recall that for $f\in \F$ and $\gp\in\Curve$, $f(\gp)$ is the image $\xi_{\gp}(f)$ of $f$ under the place $\xi_{\gp}$ associated with $\CO_{\gp}$. When we apply the map $\cto$ to $\kk(x)$ (we will then denote it by $\cto_{\kk(x)}$), we have to keep in mind that the $\kk$-rational places on $\kk(x)$ are precisely the $(x-a)$-adic places $\xi_a$ that send $x$ to $a$, for all $a\in \kk$, together with the place $\xi_\infty$ that sends $\frac{1}{x}$ to 0. If $\gp\in\Curve$ and $x(\gp)=a\in \kk$, then the restriction of $\xi_{\gp}$ to $\kk (x)$ is $\xi_a\,$.

Suppose that $(\gp,\gq)$ is an interval in $\Curve_i$ on which the function $x$ is monotonic, and take $(\CL,\CU)\in \PCuts((\gp,\gq))$. Then, writing $\cto_{\F}$ when we apply the map $\cto$ on $\F$, 
\begin{align*}
\cto_{\F}\bigl((\CL,\CU)\bigr)
&= \bigl\{f\in \F\st \exists_{\gl\in \CL \cup \{\Infty{i}\}}\exists_{\gu\in \CU \cup \{\Infty{i}\}}\forall_{\gr\in (\gl,\gu)}\ f(\gr)>0\bigr\}\\
&= \bigl\{f\in \F\st \exists_{\gl\in (\CL\cap (\gp,\gq)) \cup \{\gp\}}\exists_{\gu\in (\CU\cap (\gp,\gq)) \cup \{\gq\}}\forall_{\gr\in (\gl,\gu)}\ f(\gr)>0\bigr\}\\
&= \bigl\{f\in \F\st \exists_{\gl\in (\CL\cap (\gp,\gq)) \cup \{\gp\}}\exists_{\gu\in (\CU\cap (\gp,\gq)) \cup \{\gq\}}\forall_{\gr\in (\gl,\gu)}\ \xi_{\gr} (f)>0\bigr\}.
\end{align*}
Let us consider the case where the function $x$ is increasing on $(\gp,\gq)$; the other case is symmetrical. Set $p=\xi_{\gp}(x)\in\kk$ if $x$ has no pole in $\gp$, and $p=-\infty$ otherwise. Analogously, set $q=\xi_{\gq}(x)\in\kk$ or $q=\infty$, respectively. Then $\pi_x((\CL,\CU))\in \PCuts((p,q))$, and we denote this cut by $(L,U)$. The restriction of the above ordering to $\kk (x)$ is
\begin{align*}
\lefteqn{\mbox{}\hspace*{-2em}\cto_{\F}\bigl((\CL,\CU)\bigr)\cap \kk(x)=}\\
&=\Bigl\{f\in \kk(x)\st \exists_{\gl\in (\CL\cap (\gp,\gq)) \cup \{\gp\}}\exists_{\gu\in (\CU\cap (\gp,\gq)) \cup \{\gq\}}\forall_{\gr\in (\gl,\gu)}\ \xi_{\gr}|_{\kk(x)}(f)>0\Bigr\}\\
&=\Bigl\{f\in \kk(x)\st \exists_{l\in (L\cap (p,q)) \cup \{p\}}\exists_{\gu\in (\CU\cap (p,q)) \cup \{q\}}\forall_{r\in (l,u)}\ \xi_{r}(f)>0\Bigr\}\\
&=\Bigl\{f\in \kk(x)\st \exists_{l\in L \cup \{-\infty\}}\exists_{\gu\in \CU \cup \{+\infty\}}\forall_{r\in (l,u)}\ \xi_{r}(f)>0\Bigr\}\\
&= \cto_{\kk(x)}\bigl((L,U)\bigr)\\
&= \cto_{\kk(x)}\bigl(\pi_x\bigl((\CL,\CU)\bigr)\bigr).
\end{align*}

We have proved:

\begin{prop}\label{diagram}
The following diagram commutes:
\begin{diagram}[PostScript,small]
\Cuts(\Curve) && \rTo^{\cto_{\F}} && \Orderings(\F)\\
\dTo<{\pi_x} &&&& \dTo>{\text{\rm res}} \\
\Cuts(\kk) && \rTo^{\cto_{\kk(x)}} && \Orderings\bigl(\kk(x)\bigr)
\end{diagram}
where \textup{res} is the restriction $\ordering\mapsto \ordering\cap \kk(x)$.
\end{prop}

Now we can provide a different proof for the proposition stated by C. Scheiderer in the appendix to \cite{GBH13}:

\begin{prop}
The space $\Orderings(\F)$ is homeomorphic to $\Orderings\bigl(\kk(x)\bigr)$.
\end{prop}

\begin{proof}
Take a transcendental element $x \in K$ and  on every component of $\Curve$ choose a sequence  of points  such that  $x$ is defined and monotonic on each interval between two consecutive points. By Equation~\eqref{eq_Cuts=disjoint_union}, the set $\Cuts(\Curve)$ is a disjoint union of, say, $k$-many sets of the form $\PCuts\bigl((\gp, \gq)\bigr)$, each of which is homeomorphic to a set $\PCuts(I)$, where $I$ is an open interval in $\kk$. For any two open intervals $I_1$ and $I_2$ in $\kk$  there is a rational function $f$ which induces an order isomorphism between  $I_1$ and $I_2$, and hence the homeomorphism between $\PCuts(I_1)$ and $\PCuts(I_2)$. Now choose $k-1$ many  elements $a_1<\dotsb <a_{k-1}$ in $\kk$. By what we observed, $\Cuts(\Curve)$ is homeomorphic to the union $\PCuts\bigl((-\infty, a_1)\bigr)\, \dot\cup\,\dotsc\,\dot\cup\, \PCuts\bigl((a_{k-1},+\infty)\bigr) = \Cuts(\kk)$.
\end{proof}

Note that $\pi_x$ may fail to be a homeomorphism from $\Cuts(\Curve)$ to $\Cuts(\kk)$ as it may turn out to be neither injective nor surjective. 

In the literature (e.g.\ in \cite{Gil81}) it is common practise to use the function
\[
\psi: X(\kk(x))\ni \beta\mapsto \bigl(\{a\in\kk\st x-a\in \beta\},\{b\in\kk\st b-x\in\beta\}\bigr) \in \Cuts(\kk).
\]
We show that our function $\cto_{\kk(x)}$ is compatible with this practise:

\begin{prop} 
The functions $\psi$ and $\cto_{\kk(x)}$ are equal on $\kk(x)$.
\end{prop}

\begin{proof}
It suffices to show that $\otc\bigl(\psi(\beta)\bigr)=\beta$ for any $\beta\in X(\kk(x))$. If $l\in \{a\in\kk\st x-a\in \beta\}$, then $x-l$ has positive values on all of $(l,+\infty)$ and therefore, $x-a\in\otc \bigl(\psi(\beta)\bigr)$. Symmetrically, $u-x\in\otc\bigl(\psi(\beta)\bigr)$ for all $u\in \{b\in\kk\st b-x\in\beta\}$. It is well know that since $\kk$ is real closed, an ordering on $\kk(x)$ is uniquely determined by the signatures of the linear polynomials $x-c$, $c\in \kk$. Thus we obtain that $\otc(\psi(\beta))=\beta$.
\end{proof}

\section{Ultrametric balls}\label{sec_uballs}
Assume that $\kk$ is a real closed field and let $v$ be its natural valuation, with value group $\VG$. The field $\kk$ is an \term{ultrametric space} with the ultrametric distance
\[
d:\kk \times \kk \to \VG \cup \{\infty\},\qquad d(a,b) = v(b-a).
\]
When $\kk$ is nonarchimedean, this ultrametric distance is non-trivial. The value group $\VG$ is linearly ordered and so we can consider cuts on it, defined in a standard way. Fix a cut $(\VL, \VU)$ of $\VG$. The set 
\[
B_\VU(a) = \bigl\{b\in \kk \st d(a,b) \in \VU\cup \{\infty\}\bigr\}
\]
is called an \term{ultrametric ball} centered in $a$ with radius $\VU$. We may characterize ultrametric balls in $\kk$ by the following property: a subset $B$ is an ultrametric ball in $\kk$ if and only if 
\begin{equation} \label{eq_ballproperty}
d(a,b)>d(a,c)
\end{equation}  
for every $a,b \in B$ and $c\in \kk \setminus B$. Recall that every point of an ultrametric ball is its center and any two ultrametric balls which have a non-empty intersection are comparable by inclusion.

The structure of an ultrametric space extends naturally from $\kk$ to a finitely dimensional affine space $\AA^n\kk$ over $\kk$. It is a basic and well known fact that over the reals all the metrics of the form
\begin{align*}
d_p\bigl((x_1, \dotsc, x_n), (y_1, \dotsc, y_n)\bigr) &:= \Bigl( \sum_{i\leq n} |x_i-y_i|^p\Bigr)^{\sfrac1p},
& \text{for }p= 1,2,\dotsc\\
d_\infty\bigl((x_1, \dotsc, x_n), (y_1, \dotsc, y_n)\bigr) &:= \max_{i\leq n}\bigl\{|x_i-y_i|\bigr\}
\end{align*}
are equivalent. Switching from $\RR$ to an arbitrary, non-archimedean real closed field~$\kk$, we may define corresponding ultrametrics $d_p$ for $p\in \NN\cup\{\infty\}$ by setting
\begin{align*}
d_p\bigl((x_1, \dotsc, x_n), (y_1, \dotsc, y_n)\bigr) &:= v\Bigl( \sum_{i\leq n} |x_i-y_i|^p\Bigr)^{\sfrac1p},
& \text{for }p= 1,2,\dotsc\\
d_\infty\bigl((x_1, \dotsc, x_n), (y_1, \dotsc, y_n)\bigr) &:= \min_{i\leq n}\bigl\{v(x_i-y_i)\bigr\},
\end{align*}
where $v$ is the natural valuation of $\kk$. We claim that all of these ultrametrics are not only equivalent but actually equal.

\begin{prop}
All the ultrametrics $d_p:\AA^n\kk\times \AA^n\kk\to \VG$ are equal.
\end{prop}

\begin{proof}
Fix any $p\in \NN$. We will show that $d_p\equiv d_\infty$. Take two points $P= (x_1, \dotsc, x_n)$, $Q= (y_1, \dotsc, y_n)\in \AA^n\kk$. Then
\begin{multline*}
d_p(P,Q) 
= \frac{1}{p} v\Bigl( \sum_{i\leq n} \lvert x_i-y_i\rvert^p\Bigr) 
\geq \frac{1}{p} \min_{i\leq n}\bigl\{ v\lvert x_i-y_i\rvert^p\bigr\} \\
= \min_{i\leq n}\bigl\{ v(x_i-y_i)\bigr\} 
= d_\infty(P,Q).
\end{multline*}
This gives us one inequality. To prove the opposite one, recall that the natural valuation is weakly decreasing on $\kk_+\cup \{0\}$. Now, all $\lvert x_j-y_j\rvert$ are clearly non-negative. Since for every $j\leq n$ we have $\lvert x_j-y_j\rvert^p\leq \sum_{i\leq n} \lvert x_i-y_i\rvert^p$, it follows that
\[
v(x_j-y_j) 
= \frac{1}{p}v\bigl(\lvert x_j-y_j\rvert^p\bigr) 
\geq \frac{1}{p}v\Bigl( \sum_{i\leq n} \lvert x_i-y_i\rvert^p\Bigr) 
= d_p(P,Q).
\]
In particular, $d_\infty(P,Q)= \min\limits_{j\leq n}\bigl\{ v(x_j-y_j)\bigr\}\geq d_p(P,Q)$, as claimed.
\end{proof}

\section{Ball cuts and $\RR$-places}\label{sec_ball_cuts}
In this section we prove our second main result, namely a criterion for two orderings of $\F$ to be associated with a single $\RR$-place of $\F$. To this end we need first to introduce a notion of ball cuts on a curve. Let $v$ be again the natural valuation of~$\kk$. Any ultrametric ball $\ball$ in $\kk$ determines two cuts of~$\kk$, namely:
\begin{align*}
\ball^- = \bigl( \{a\in \kk\st a< \ball\}, \up\ball \bigr),\\
\ball^+ = \bigl( \lo\ball, \{a\in \kk\st a> \ball\}\bigr).
\end{align*}
The cuts $\ball^+$, $\ball^-$ are called \term{ball-cuts} in $\kk$. In particular, the cuts $(\kk,\emptyset)$ and $(\emptyset, \kk)$ are ball-cuts, since they correspond to the ultrametric ball $\ball = \kk$. Notice also that for every $b\in \kk$, the principal cuts $\bigl( \{a\in \kk\st a< b\}, \{a\in \kk\st a\geq b\}\bigr)$ and $\bigl( \{a\in \kk\st a\leq b\}, \{a\in \kk\st a> b\}\bigr)$ are ball-cuts, too. They correspond to the ultrametric ball $\{b\}$.

We have the following characterization of ball cuts in $\kk$.

\begin{prop}
Take a cut $C$ in $\kk$ and the  corresponding  ordering $\ordering$ of $\kk(x)$. Let $v_\ordering$ be the natural valuation of the rational function field~$\kk(x)$ associated with~$\ordering$, with value group $v_\ordering \kk(x)$. Then
\[
C \text{ is a ball cut} \iff \bigl[v_\ordering \kk(x):2v_\ordering \kk(x)\bigr] = 2.
\] 
\end{prop}

\begin{proof}
By the Baer-Krull Theorem (see \cite[Corollary 3.11]{Lam83}), the number of orderings compatible with the valuation $v_\ordering$ (i.e., determining the same $\mathbb R$-place of~$\kk(x)$) is equal to the order of the character group $\bigl(\sfrac{v_\ordering \kk(x)}{2v_\ordering \kk(x)}\bigr)^*$. Since  $v_\ordering \kk$ is divisible, the rational rank of $v_\ordering \kk(x)$ is at most 1, by the Abhyankar Inequality. Therefore 
$\sfrac{v_\ordering \kk(x)}{2v_\ordering \kk(x)}$ is a finite group, so the order of its character group is equal to the index $\bigl[v_\ordering \kk(x):2v_\ordering \kk(x)\bigr]$. By \cite[Theorem 2.1]{KMO11}, 
two distinct orderings of $\kk(x)$ determine the same $\mathbb R$-place if and only if the corresponding cuts of $\kk$ are the cuts at the upper and lower edge of the same ultrametric ball in $\kk$.
\end{proof}

Take a finite extension ${(F,v)} \subseteq {(E,v)}$ of valued fields. It is known (see e.g. \cite[\S3]{Kne73}), that
\[
[vF:2vF] = [vE:2vE].
\]
Now let us get back to our algebraic function field $\F$. Take a cut $C$ of the curve $\Curve$ and the associated ordering $\ordering\in \Orderings(\F)$. Further, denote by $v_\ordering$ the natural valuation of $\F$ induced by $\ordering$ and let $\place_\ordering$ be its canonical $\RR$-place. Fix an element $x\in \F$ transcendental over $\kk$. 

Suppose that $\pi_x(C)$ is a ball cut in $\kk$. It follows that
\[
2 = \bigl[ v_{\ordering\cap \kk(x)}\kk(x) : 2v_{\ordering\cap \kk(x)}\kk(x)\bigr] 
= [v_\ordering\F : 2v_\ordering\F],
\]
with the right hand side being evidently independent of $x$. This allows us to introduce the notion of ball cuts of a curve.

\begin{df}
A cut $C$ of the curve $\Curve$ is called a \term{ball cut} if for one (or equivalently, every) $x\in \F$ transcendental over $\kk$, the projection $\pi_x(C)$ is a ball cut in $\kk$.
\end{df}

Observe that in particular all the principal cuts of $\Curve$ (cf.\ Definition~\ref{df_principal_cut}) are ball cuts. Indeed, for every $\gp\in \Curve_i$, the associated principal cuts map to the ball cuts generated by singletons $x(\gp)$ in $\kk$ or to improper cuts if $x \in K$ has a pole at $\gp$. 

Now we are ready to present our second main result, that determines when two cuts, or equivalently orderings (in view of Theorem~\ref{thm_Cuts_homeo_Orderings}), correspond to the same $\RR$-place of the function field $\F$.

\begin{thm}\label{thm_2_cuts_1_place}
Let $C_1$ and $C_2$ be two ball cuts on $\Curve$. The corresponding orderings determine the same $\RR$-place of $K$ if and only if for every $x\in \F\setminus \kk$ the cuts $\pi_x(C_1)$ and $\pi_x(C_2)$ are induced by the same ultrametric ball.
\end{thm}

\begin{proof}
Take two cuts $C_1, C_2\in \Cuts(\Curve)$. Let $\ordering_1, \ordering_2\in \Orderings(\F)$ be the corresponding orderings of $\F$ and $\place_1, \place_2$ the associated $\RR$-places. Suppose that $\place_1\neq \place_2$, therefore there is an element $x\in \F$ such that $\place_1(x) \neq \place_2(x)$. Passing to the rational function field $\kk(x)$, we have $\place_1\restrict_{\kk(x)}\neq \place_2\restrict_{\kk(x)}$. It follows from  \cite[Theorem~2.2]{KMO11} that $\pi_x(C_1)$ and $\pi_x(C_2)$ cannot be induced by the same ultrametric ball.

Conversely, suppose that there is  an element $x\in \F\setminus \kk$ such that $\pi_x(C_1), \pi_x(C_2)$ do not correspond to a single ultrametric ball. Then, in particular, the associated $\RR$-places $\place_1, \place_2$ of $\F$ must differ on $x$.
\end{proof}

\section{Ball cuts of affine curves}
So far, we have been working in a general setup with some abstract real curve. However, once we embed our curve in an affine space (possibly of a high dimension) we obtain a clearer picture. In particular it turns out that the ball cuts of $\Curve$ are determined by ultrametric balls in the affine space $\AA^n\kk$. From now on when we write $\gp$ for point in  $\AA^n\kk$ it does not mean that it has to belong to $\Curve$.

\begin{prop}
Fix a smooth and complete real affine  curve $\Curve\subset \AA^n\kk$. Take a point $\gp = (x_1, \dotsc, x_n)\in \AA^n\kk$ and a cut $ (\VL, \VU)$ in $v\kk$. Assume that there is a component  $\Curve_i$ of the curve, which has non-empty intersections with both an ultrametric ball $\ball_\VU( \gp ) \in \AA^n\kk$ and its complement $\ball_\VU( \gp )^c = \AA^n\kk\setminus \ball_\VU( P )$. Then  $\ball_\VU( \gp )$ induces a ball cut on $\Curve_i$ \textup(possibly more than one\textup). 
\end{prop}

\begin{proof}
If the ultrametric ball $\ball_{\VU}(\gp)$ is a singleton $\{ \gp\}$, then $\gp$ must lie on the curve~$\Curve$ and the assertion holds trivially---the cuts induced by $\ball_\VU(\gp)$ are precisely the two principal cuts associated with~$\gp$. Hence, for the rest of the proof, we may assume that $\ball_\VU(\gp)$ consists of more than just one point. Define a function $\rho_{\gp} : \AA^n\kk\to \kk$ by the formula
\[
\rho_P(\gq) := \sum_{i=1}^n (x_i - y_i)^2,
\qquad\text{where}\quad
\gq = (y_1, \dotsc, y_n) \in \AA^n\kk.
\]
Then $\rho_{\gp} $ induces a function $r_{\gp}  \in \F$. Note that 
\[
\ball_\VU( {\gp}  ) = \bigl\{ \gq\in \AA^n\kk\st \tfrac{1}{2}v(r_{\gp} (\gq))
\in \VU\cup \{\infty\}\bigr\}.
\]
Note that $r_{\gp} $ is defined everywhere on the curve. Proposition~\ref{prop_piecewise_monotonic} asserts that there are points $\gp_1 \prec \dotsc \prec\gp_m\in \Curve_i$ such that $r_{\gp} $ is   monotonic on every interval $(\gp_i, \gp_{i+1})$, and hence also on $[\gp_i, \gp_{i+1}]$ (by the Intermediate Value Theorem). This yields  that if $\gp_i$ and $\gp_{i+1}$ are in $\ball_{\VU}(\gp)$, then the whole interval  $[\gp_i, \gp_{i+1}]$ lies in $\ball_{\VU}(\gp)$. Since $\ball_{\VU}(\gp)$ cuts through the component $\Curve_i$,  there is at least one  interval $[\gp_i, \gp_{i+1}]$ that has nonempty intersections with both $\ball_\VU( \gp ) $ and $\ball_\VU( \gp )^c$. Set either
\[
\CL := [\gp_i, \gp_{i+1}]\cap \ball_{\VU}(\gp)
\qquad\text{and}\qquad
\CU := [\gp_i, \gp_{i+1}]\cap \ball_{\VU}(\gp)^c
\]
when $r_{\gp} $ is increasing on $(\gp_i, \gp_{i+1})$, or
\[
\CL := [\gp_i, \gp_{i+1}]\cap \ball_{\VU}(\gp)^c
\qquad\text{and}\qquad
\CU := [\gp_i, \gp_{i+1}]\cap \ball_{\VU}(\gp)
\]
when $r_{\gp} $ is decreasing. Using the definition of the function $r_{\gp} $ we see that all the elements of $\CL$ precede all the elements of $\CU$, due to the fact that $r_{\gp} $ is monotonic on $[\gp_i, \gp_{i+1}]$. Thus, the pair $(\CL, \CU)$ constitutes a proper cut of $[\gp_i, \gp_{i+1}]$. This cut is induced by a cut of the whole component $\Curve_i$ as shown in Section \ref{sec_cuts}. The projection $\pi_{r_{\gp} }$ projects our cut to the ball cut  $B_{2U}(0)^+$, therefore it is indeed a ball cut.
\end{proof}

\begin{thm}
Let $\Curve$ be a smooth and complete real affine curve. Then every ball cut on $\Curve$ is induced by some ultrametric ball in $\AA^n\kk$.
\end{thm} 

\begin{proof}
First observe that for every component $\Curve_i\subseteq \Curve$, the assertion holds for the principal cuts $\locut{\gp}$ and $\upcut{\gp}$. Indeed, both cuts are induced by the ultrametric ball~$\{\gp\}$. 

Now let $C = (\CL, \CU)$ be a non-principal ball cut of a component $\Curve_k\subseteq \Curve$. For every coordinate $x_i$, $i\in \{1,\dotsc ,n\}$, the projection $\pi_i(C) = \pi_{x_i}(C)$ is a ball cut of~$\kk$. Choose an interval $[\gp,\gq]$ on  $\Curve_k$ in such a way that for every $i\in \{1,\dotsc, n\}$:
\begin{itemize}
\item $C$ is a proper cut of $[\gp,\gq]$,
\item $x_i$ is monotonic on $(\gp,\gq)$,
\item the ultrametric ball in $\kk$ determining $\pi_i(C)$ is centered in either $p_i := x_i(\gp)$ or  $q_i := x_i(\gq)$.
\end{itemize}
The choice is possible by  means of Proposition~\ref{prop_piecewise_monotonic} and the fact that every point of an ultrametric ball is its center. Taking $\gp$ or $\gq$ to be an inverse image of an element in the ultrametric ball generating $\pi_i(C)$ one achieves that the last condition holds. Since we are working in a finite-dimensional space, we may narrow the interval so that the above three conditions are simultaneously met for all coordinates. Denote by $\Balls_{\gp}$ (respectively, $\Balls_{\gq}$) the set of all ultrametric balls generated by $\pi_i(C)$ and centered in $p_i$ or $q_i$ for $i\in \{1, \dotsc, n\}$. Possibly one of these two sets can be empty. The radii of ultrametric balls are (by definition) upper cut sets. Let~$T_{\gp}$ (respectively, $T_\gq$) be the union of all the radii of the balls in $\Balls_{\gp}$ (respectively,~$\Balls_\gq$). 

We claim that $T_{\gp} \neq T_{\gq}$. We assume that both sets  $\Balls_\gp$, $\Balls_\gq$ are non-empty, because otherwise the claim is clearly true. Assume, in order to derive a contradiction, that $T_{\gp} =T_{\gq}$. Denote this set just by~$T$. Without loss of generality we can assume that $\pi_1(C)$ is a ball cut of a ball $B_T(p_1)$ and $\pi_2(C)$ is a ball cut of a ball $B_T(q_2)$. Consider the function $z = x_1 + \varepsilon x_2$, where $\varepsilon \in \{\pm 1\}$ is selected in such a way that both $x_1$ and $\varepsilon x_2$ are increasing or both are decreasing on $(\gp,\gq)$.
This implies that $z$ is monotonic on $(\gp,\gq)$. The projection $\pi_z(C)$ is a ball cut of a ball $\ball$ centered in $z(\gl)$ for some $\gl \in [\gp,\gq]$. Assume that $\gl$ lies in the lower cut set $\CL$ of the interval $[\gp,\gq]$. For $\gl$ lying in the upper cut set $\CU$ the proof is symmetric. For any $\gu \in \CU \cap [\gp,\gq]$ we have that 
\[
d(l_2,u_2) = d(l_2, q_2),
\]
because $d(l_2,q_2)< d(u_2, q_2)$. On the other hand, since $\ball_T(p_1)$ and $\ball_T(q_2)$ have the same radii and $d(l_2, u_2)$ is fixed, we can choose $\gu$ to be so close to the cut $C$ that 
\[
d(l_1,u_1)\geq d(l_2, u_2).
\]
Now take another point $\gl' \in \CL \cap (\gl,\gu)$ such that $ d(l_2,l'_2)= d(l_2, u_2)$. The point $\gl'$ can be, for instance, chosen as a point on $\Curve_k$ with $l'_2 = \sfrac{(l_2+ u_2)}{2}$. Since $v(l_1 - l'_1) \in T$ and $v(l_2 - l'_2)<T$, we have that
\begin{multline*}
d\bigl(z(\gl), z(\gl')\bigr) = v\bigl(z(\gl)-z(\gl')\bigr) 
= v\bigl((l_1 + \varepsilon l_2) - (l'_1 + \varepsilon l'_2)\bigr)\\ 
= v\bigl((l_1 - l'_1) + \varepsilon (l_2 - l'_2)\bigr) 
= v(l_2 - l'_2) =  d(l_2 , l'_2) = d(l_2, u_2).
\end{multline*}
On the other hand, $(l_1 - u_1)$ and $\varepsilon (l_2 - u_2)$ have the same signs (because of our choice of $\varepsilon$), and we obtain that
\begin{multline*}
d\bigl( z(\gl), z(\gu)\bigr) = v\bigl(z(\gl) - z(\gu)\bigr) 
= v\bigl( (l_1 + \varepsilon l_2) - (u_1 + \varepsilon u_2)\bigr) \\
= v\bigl( (l_1 - u_1) + \varepsilon (l_2 - u_2))
= v(l_1 - u_1) = d(l_1 , u_1).
\end{multline*}
Consequently, 
\[
d\bigl(z(\gl), z(\gl')\bigr) \leq  d\bigl(z(\gl), z(\gu)\bigr),
\]
but this contradicts \eqref{eq_ballproperty}. The inequality $T_{\gp} \neq T_{\gq}$ is proved.

In order to complete the proof of the theorem, consider two cases: $T_{\gq} \subset T_{\gp}$ and $T_{\gp} \subset T_{\gq}$. In the first case take all  projections $\pi_i$ sending the cut $C$ to a ball cut of a ball in $\Balls_{\gq}$. Every ball in $\Balls_{\gq}$ has a radius strictly contained in $T_{\gp}$, thus for every such ball we find an element $l_i$ in the lower cut set of $\pi_i(C)$ but strictly greater than $p_i$ and so close to $q_i$ that $v(l_i -q_i) \in  T_{\gp}$. Of all the preimages $\bigl\{ x_i^{-1}(l_i) \bigr\}$ select the one which is closest to $C$ (i.e., the interval $(x_i^{-1}(l_i), \gq)$ is smallest in the sense of inclusion), and call it $\gl$. Take a ball $\ball_{T_{\gp}}(\gl)$. Note that for every $\gl' \in \CL \cap (\gl,\gq)$ we have that
\[
d(\gl,\gl') =  \min\bigl\{v(l_1 - l'_1),\dotsc, v(l_m - l'_m)\bigr\} \in T_{\gp}.
\]
For every $\gu \in \CU \cap (\gl,\gq)$ we have $d(\gl,\gu) < T_{\gp},$ because $v(l_i - u_i)< T_{\gp} $ for the projection $\pi_i$ which realizes the radius $T_{\gp}$. Therefore the ball $\ball_{T_{\gp}}(\gl)$ induces the cut $C$ on $\Curve$.

For the second case the proof is symmetric.
\end{proof}

Using Theorem~\ref{thm_2_cuts_1_place} we obtain now:

\begin{thm}
Assume that $\Curve\subset \AA^n\kk$ is a smooth and complete real affine curve. Let $C_1$ and $C_2$ be two ball cuts on $\Curve$. If the corresponding orderings determine the same $\RR$-place of $\F$, then there is an ultrametric ball $\ball\subset \AA^n\kk$ inducing $C_1$ and $C_2$ on $\Curve$.
\end{thm}

\begin{proof}
Assume that the orderings corresponding to $C_1$ and $C_2$ determine the same place.  Assume that $C_1$ is a cut on the component $\Curve_k$ determined by an ultrametric ball $B_U(P)$ for some $P=(x_1,...,x_n)$ in $\AA^n\kk$ (we may assume that $P\in \Curve_k$).
Then  the function  $\rho_P : \AA^n\kk\to \kk$ given by 
\[
\rho_P(Q) := \sum_{i=1}^n (x_i - y_i)^2,
\qquad\text{where}\quad
Q = (y_1, \dotsc, y_n) \in \AA^n\kk.
\] 
is not constant. Let $r_P$ be the polynomial function on $\Curve$ induced by $\rho_P$. The projection $\pi_{r_P}$ sends $C_1$ to the cut $B_{2U}(0)^+$. Theorem~\ref{thm_2_cuts_1_place} asserts that $\pi_{r_P}(C_2) = B_{2U}(0)^+$, as well (because $r_P$ has only non-negative values). Therefore the cut $C_2$ is also determined by $\ball_U(P)$.
\end{proof}

\subsection*{Example} 
Let $\kk$ be a nonarchimedean real closed field and $I_{\kk}$ be the ideal of infinitesimals in $\kk$, which is the maximal ideal of the valuation ring of the natural valuation $v$ of $\kk$. Take $a \in I_{\kk}$ and consider the function field $\F$ of the genus two plane curve $\Curve$ given by the equation
\begin{equation} \label{eq_genus2}
y^2 + (x^2-a^2)(x^2-1) = 0.
\end{equation}
The curve has two components which can be separated by the function $x\in \F$. Let $T$ be the set of positive elements in the value group $\VG$. Consider the ultrametric ball centered in the origin $\go$: 
\[
\ball_T(\go) = I_{\kk}\times I_{\kk}.
\]
The ball $\ball_T(\go)$ determines four cuts on $\Curve$, two on each component. For each of these cuts we consider the signatures of the coordinates $x$ and $y$ in the orderings corresponding to these cuts. The four cuts can be distinguished by these signatures, so let us denote the cuts by $C(\sgn_x, \sgn_y)$. The images of all four cuts under the  projections $\pi_x$ and $\pi_y$  are the cuts determined by the ball $\ball_T(0)$ in $\kk$.
 
Take $z = \sfrac{y}{x}$ and consider the projection $\pi_z$. Equation~\eqref{eq_genus2} can be rewritten as  
\[
z^2x^2 + \bigl(x^2-a^2\bigr)\bigl(x^2-1\bigr) = 0.
\]
Therefore, 
\[
0 = z^2 + \biggl(1-\Bigl(\frac{a}{x}\Bigr)^2 \biggr)\bigl(x^2-1\bigr) = z^2-1+x^2 + \Bigl(\frac{a}{x}\Bigr)^2 -a^2. 
\]
Thus,
\begin{equation}\label{equation_z}
(z+1)(z-1) = a^2 - x^2-\Bigl(\frac{a}{x}\Bigr)^2.
\end{equation}
Let $v$ be a real valuation of $\F$ such that $v(a) > v(x) > 0$. Then 
\[
v\bigl((z-1)(z+1)\bigr)  = v\biggl(a^2 - x^2-\Bigl(\frac{a}{x}\Bigr)^2\biggr) >0,
\]
so
\[
z \in B_T(1)\quad \text{or}\quad z \in B_T(-1).
\]
Now assume that $v(x)=0$. Then, 
\[
v\bigl((z-1)(z+1)\bigr) = v(x^2) = 0.
\]
Note that if  $v(z+1) > 0$, then $v(z-1) = v(z+1-2)= v(-2) = 0$. Symmetrically, if $v(z-1) > 0$, then $v(z+1) = 0$. Therefore, $v(z+1) = v(z-1) = 0$ which shows that $z \notin B_T(1)$ and $z \notin \ball_T(-1)$. Further we see that the right hand side of Equation~\eqref{equation_z}, up to an infinitesimal, equals $-x^2$ and hence is negative. Therefore, $-1<z<1$, which means that $z$ lies between the balls $\ball_T(1)$ and $\ball_T(-1)$. 
The sign of $z$ depends on the signs of $x$ and $y$, and thus,
\[
\pi_z\bigl(C(1,1)\bigr) =\pi_z\bigl(C(-1,-1)\bigr) =  \ball_T(1)^-
\]
and 
\[
\pi_z\bigl(C(-1,1)\bigr) =\pi_z\bigl(C(1,-1)\bigr) =  \ball_T(-1)^+,
\]
so the orderings corresponding to the cuts $C(1,1)$ and $C(-1,-1)$ give the same $\RR$-place which is different from the place determined by the orderings corresponding to $C(-1,1)$ and $C(1,-1)$.

\bibliographystyle{alpha} 
\bibliography{places}

\begin{thebibliography}{KMO11}

\bibitem[BCR98]{BCR98}
Jacek Bochnak, Michel Coste, and Marie-Fran{\c{c}}oise Roy.
\newblock {\em Real algebraic geometry}, volume~36 of {\em Ergebnisse der
  Mathematik und ihrer Grenzgebiete (3) [Results in Mathematics and Related
  Areas (3)]}.
\newblock Springer-Verlag, Berlin, 1998.

\bibitem[GBH13]{GBH13}
Nicolas Grenier-Boley and Detlev~W. Hoffmann.
\newblock Isomorphism criteria for {W}itt rings of real fields.
\newblock {\em Forum Math.}, 25(1):1--18, 2013.
\newblock With an appendix by Claus Scheiderer.

\bibitem[Gil82]{Gil81}
Robert Gilmer.
\newblock Extension of an order to a simple transcendental extension.
\newblock {\em Contemp.Math.}, 8:113--118, 1982.

\bibitem[KMO11]{KMO11}
Franz-Viktor Kuhlmann, Micha{\l} Machura, and Katarzyna Osiak.
\newblock Metrizability of spaces of {$\mathbb{R}$}-places of function fields
  of transcendence degree 1 over real closed fields.
\newblock {\em Comm. Algebra}, 39(9):3166--3177, 2011.

\bibitem[Kne73]{Kne73}
Manfred Knebusch.
\newblock On the extension of real places.
\newblock {\em Comment. Math. Helv.}, 48:354--369, 1973.

\bibitem[Kne76a]{Kne76a}
Manfred Knebusch.
\newblock On algebraic curves over real closed fields. {I}.
\newblock {\em Math. Z.}, 150(1):49--70, 1976.

\bibitem[Kne76b]{Kne76b}
Manfred Knebusch.
\newblock On algebraic curves over real closed fields. {II}.
\newblock {\em Math. Z.}, 151(2):189--205, 1976.

\bibitem[Lam83]{Lam83}
Tsit~Yuen Lam.
\newblock {\em Orderings, valuations and quadratic forms}, volume~52 of {\em
  CBMS Regional Conference Series in Mathematics}.
\newblock Published for the Conference Board of the Mathematical Sciences,
  Washington, DC, 1983.

\bibitem[Pre84]{Pre84}
Alexander Prestel.
\newblock {\em Lectures on formally real fields}, volume 1093 of {\em Lecture
  Notes in Mathematics}.
\newblock Springer-Verlag, Berlin, 1984.

\end{thebibliography}
\end{document}